\documentclass[10pt,a4paper]{article}
\usepackage[utf8]{inputenc}
\usepackage{amsmath}
\usepackage{amsfonts}
\usepackage{amssymb}
\author{Andreas Basse-O'Connor, Mikkel Slot Nielsen,
Jan Pedersen, \\ and Victor Rohde \\
Department of Mathematics \\
Aarhus University \\
\{basse, mikkel, jan, victor\}@math.au.dk
}

\usepackage{amssymb}
\usepackage{amsthm}
\usepackage{amscd}
\usepackage{amstext}
\usepackage{fancyhdr}
\usepackage{amsthm}
\usepackage[numbers]{natbib}
\usepackage[T1]{fontenc}
\usepackage{enumerate}
\usepackage[american]{babel}
\usepackage{graphicx}
\usepackage{bbm}
\usepackage[mathscr]{euscript}
\usepackage{mathrsfs}
\usepackage{stmaryrd}
\usepackage{soul}
\usepackage{hyperref}
\usepackage{xspace}
\usepackage[draft,danish]{fixme}
\usepackage{mathtools}
\usepackage{dsfont}
\usepackage{url}
\usepackage[ps,all,dvips]{xy}
\SelectTips{cm}{12}
\CompileMatrices

  \makeatletter
 \useshorthands{"}
 \defineshorthand{"-}{\nobreak-\bbl@allowhyphens}
 \makeatother

\DeclareMathOperator{\Real}{Re}
 
\newcommand{\FF}{\mathcal{F}} 
\newcommand{\PP}{\mathbb{P}} 
\newcommand{\EE}{\mathbb{E}} 

\newcommand{\NN}{\ensuremath{\mathbb{N}}}           % The natural numbers
\newcommand{\RR}{\ensuremath{\mathbb{R}}}           % The real numbers
\newcommand{\CC}{\ensuremath{\mathbb{C}}}

\newcommand{\LL}{\mathcal{L}}

\theoremstyle{plain}
\newtheorem{lemma}{Lemma}[section]

\newtheorem{theorem}[lemma]{Theorem}
\newtheorem{corollary}[lemma]{Corollary}
\theoremstyle{definition}
\newtheorem{example}[lemma]{Example}

\theoremstyle{definition}

\theoremstyle{remark}

\newtheorem{remark}[lemma]{Remark}

\newcommand{\devnull}[1]{}

\numberwithin{equation}{section}

\date{\, }

\begin{document}

%\begin{frontmatter}

% "Title of the Paper"
\title{Continuous-time models with an autoregressive structure}
%\runtitle{A continuous-time framework for ARMA processes}

%\begin{aug}
% indicate corresponding author with \corref{}
% \author{\fnms{John} \snm{Smith}\thanksref{a}\corref{}\ead[label=e1]{smith@foo.com}\ead[label=e2,url]{www.foo.com}}
% \address[a]{\printead{e1};\printead{e2}}

%\author{\fnms{Andreas} \snm{Basse-O'Connor,}\thanksref{a}\ead[label=e1]{\{basse, mikkel, jan, victor\}@math.au.dk}}
%\author{\fnms{Mikkel Slot} \snm{Nielsen,}\thanksref{a}\ead[label=e2]{mikkel@math.au.dk}}
%\author{\fnms{Jan} \snm{Pedersen,}\thanksref{a}\ead[label=e3]{jan@math.au.dk}}
%\and
%\author{\fnms{Victor} \snm{Rohde}\thanksref{a}\ead[label=e4]{victor@math.au.dk}}
%\runauthor{A. Basse-O'Connor, M. Nielsen, J. Pedersen, and V. Rohde}

%\address[a]{Department of Mathematics, Aarhus University, Ny Munkegade 118, 8000 Aarhus C, Denmark. \printead{e1}}
%\address[b]{\printead{e2}}

%\affiliation{Aarhus University}

%\end{aug}

\maketitle

\begin{abstract}

In this paper we suggest two continuous-time models which exhibit an autoregressive structure. We obtain existence and uniqueness results and study the structure of the solution processes. One of the models, which corresponds to general stochastic delay differential equations, will be given particular attention. We use the obtained results to link the introduced processes to both discrete-time and continuous-time ARMA processes. 
\smallskip


%\noindent \footnotesize \textit{AMS 2010 subject classifications:} 

\end{abstract}

%\begin{keyword}
%\kwd{continuous-time ARMA model}
%\kwd{stochastic delay differential equations}
%\kwd{processes of Ornstein-Uhlenbeck type}
%\kwd{long-range dependence}
%\kwd{CARMA processes}
%\kwd{moving averages}
%\end{keyword}

% history:
% \received{\smonth{1} \syear{0000}}

%\tableofcontents

%\end{frontmatter}

\section{Introduction}

Let $(L_t)_{t\in \mathbb{R}}$ be a two-sided L\'{e}vy process and $\psi:\mathbb{R}\to \mathbb{R}$ some measurable function which is integrable with respect to $(L_t)_{t\in \mathbb{R}}$ (in the sense of \cite{Rosinski_spec}). Processes of the form
\begin{align}\label{MA}
X_t = \int_\RR \psi (t-u) \, dL_u,\quad t \in \mathbb{R},
\end{align}
are known as (stationary) continuous-time moving averages and have been studied extensively. For instance, this class nests the discrete-time moving average with filter $(\psi_j)_{j\in \mathbb{Z}}$ (at least when it is driven by an infinitely divisible noise), since one can choose $\psi (t) = \sum_{j \in \mathbb{Z}} \psi_j \mathds{1}_{(j-1,j]}(t)$. Another example of (\ref{MA}) is the Ornstein-Uhlenbeck process corresponding to $\psi (t) = e^{-\lambda t}\mathds{1}_{[0,\infty)}(t)$ for $\lambda >0$. Ornstein-Uhlenbeck processes often serve as building blocks in stochastic modeling, e.g. in stochastic volatility models for option pricing as illustrated in \cite{OEBN_non_G_OU} or in models for (log) spot price of many different commodities, e.g., as in \cite{schwartz1997stochastic}. A generalization of the Ornstein-Uhlenbeck process, which is also of the form \eqref{MA}, is the CARMA process. To be concrete, for two real polynomials $P$ and $Q$, of degree $p$ and $q$ ($p>q$) respectively, with no zeroes on $\{z \in \mathbb{Z}\, :\, \text{Re}(z)=0\}$, choosing $\psi:\mathbb{R}\to \mathbb{R}$ to be the function characterized by
\begin{align*}
\int_\mathbb{R} e^{-ity}\psi (t)\, dt = \frac{Q(iy)}{P(iy)},\quad y \in \mathbb{R},
\end{align*}
results in a CARMA process. CARMA processes have found many applications, and extensions to account for long memory and to a multivariate setting have been made. For more on CARMA processes and their extensions, see \cite{BrockwellMarquardt,BrockwellLindner,JonesAckerson,MarquardtStelzer,TodorovTauchen}. In general, properties of continuous-time moving averages, such as when they have long memory and have sample paths of finite variation or, more generally, are semimartingales, are well understood. For an extensive treatment of these processes we refer to \cite{characterBasse,abocjr} and references therein.  

Instead of specifying the kernel $\psi$ in (\ref{MA}) directly it is often preferred to view $(X_t)_{t\in \mathbb{R}}$ as a solution to a certain equation. For instance, as an alternative to (\ref{MA}), the Ornstein-Uhlenbeck process with parameter $\lambda>0$, respectively the discrete-time moving average with filter $\psi_j = \alpha^j\mathds{1}_{j\geq 1}$ for some $\alpha\in \mathbb{R}$ with $\vert \alpha \vert <1$, may be characterized as the unique stationary process that satisfies
\begin{align}\label{OUequation}
dX_t = -\lambda X_t\, dt + dL_t,\quad t \in \mathbb{R},
\end{align}
respectively, 
\begin{align}\label{ARequation}
X_t = \alpha X_{t-1} + L_t - L_{t-1},\quad t \in \mathbb{R}.
\end{align}
The representations \eqref{OUequation}-\eqref{ARequation} are useful in many aspects, e.g., in the understanding of the evolution of the process over time and in the study of the properties of $(L_t)_{t\in \mathbb{R}}$ through observations of $(X_t)_{t\in\mathbb{R}}$. This motivates generalizing the equations \eqref{OUequation}-\eqref{ARequation} and studying the corresponding solutions.

\medskip

\noindent \textbf{The two models of interest:} Let $\eta$ and $\phi$ be finite signed measures concentrated on $[0,\infty)$ and $(0,\infty)$, respectively, and let $\theta:\mathbb{R}\to \mathbb{R}$ be some measurable function (typically chosen to have a particularly simple structure) which is integrable with respect to $(L_t)_{t\in \mathbb{R}}$. Moreover, suppose that $(Z_t)_{t\in \mathbb{R}}$ is a measurable and integrable process with stationary increments. The equations of interest are
\begin{align}\label{SDDE}
dX_t  = \Bigr( \int_{[0,\infty)} X_{t-u} \,\eta(du) \Bigr) \, dt + dZ_t,\quad t\in \mathbb{R},
\end{align}
and
\begin{align}\label{levelsCARMA}
X_t =   \int_0^\infty X_{t-u} \,\phi(du) +  \int_{-\infty}^t \theta(t-u) \,dL_u,\quad t \in \mathbb{R}.
\end{align}
%The idea is that the term
%\begin{align*}
%\int_0^\infty X_{t-u} \,\phi(du), \quad \text{respectively} \quad \int_\mathbb{R} X_u \phi_{s,t} (du),
%\end{align*}
%can be seen as an autoregressive, respectively increment autoregressive term, and therefore that the models \eqref{levelsCARMA} and \eqref{IncCARMA} may be interpreted as a continuous-time framework for ARMA processes. Besides investigating which specifications of $\phi$ that lead to a stationary solution and when it is unique, we will also be interested in the structure and properties of the solution process.
We see that \eqref{OUequation} is a special case of \eqref{SDDE} with $\eta = - \lambda \delta_0$ and $Z_t = L_t$, and \eqref{ARequation} is a special case of \eqref{levelsCARMA} with $\phi = \alpha \delta_1$ and $\theta = \mathds{1}_{(0,1]}$. (Here $\delta_c$ refers to the Dirac measure at $c\in \mathbb{R}$.) Equation \eqref{SDDE} is known in the literature as a stochastic delay differential equation (SDDE), and existence and (distributional) uniqueness results have been obtained when $\eta$ is compactly supported and $(Z_t)_{t\in \RR}$ is a Lévy process (see \cite{GK,KuMen}). The case with a general noise $(Z_t)_{t\in \mathbb{R}}$ has, to the best of our knowledge, only been studied in \cite{Mohammed} where the delay measure $\eta$ is assumed to have compact support. Another generalization of the noise term is given in \cite{ReissRiedle}. Other parametrizations of $\phi$ in \eqref{levelsCARMA} that we will study in Example~\ref{expMeasure} and \ref{discreteCond} are 
\begin{equation*}
\phi (du) = \alpha e^{-\beta u} \mathds{1}_{[0,\infty)}(u)\, du
\end{equation*}
for $\alpha \in \mathbb{R}$ and $\beta >0$ and 
\begin{equation*}
\phi  = \sum_{j=1}^p \phi_j \delta_j
\end{equation*}
for $\phi_j \in \mathbb{R}$. As far as we know, equations of the type \eqref{levelsCARMA} have not been studied before. We will refer to \eqref{levelsCARMA} as a level model, since it specifies $X_t$ directly (rather than its increments, $X_t-X_s$). Although the level model may seem odd at first glance as the noise term is forced to be stationary, one of its strengths is that it can be used as a model for the increments of a stationary increment process. We present this idea in Example~\ref{SDDEstatInc} where a stationary increment solution to \eqref{SDDE} is found when no stationary solution exists.

%For the autoregressive term in \eqref{IncCARMA}, an important subclass is when $\phi$ has a density $f$ with respect to the Lebesgue measure, such that $f$ induces a finite Lebesgue-Stieltjes measure $\eta$. In this case, $\phi_{s,t}(du) = \int_\mathbb{R} \mathds{1}_{(s,t]}(u+v) \,\eta (dv) du$, and \eqref{IncCARMA} can be written as

\medskip

\noindent \textbf{Our main results:} In Section~\ref{SDDEsection} we state a general existence and uniqueness result for the model \eqref{SDDE} and provide several examples of choices of $\eta$ and $(Z_t)_{t\in \mathbb{R}}$. Among other things we show that long memory (in the sense of a hyperbolically decaying autocovariance function) must be incorporated through the noise process $(Z_t)_{t\in \mathbb{R}}$, and we indicate how invertible CARMA processes can be viewed as solutions to SDDEs. Moreover, in Corollary~\ref{simmaNoise} it is observed that as long as $(Z_t)_{t\in \mathbb{R}}$ is of the form
\begin{align*}
Z_t = \int_\mathbb{R} \big[\theta (t-u) - \theta_0 (-u) \big]\, dL_u,\quad t \in \mathbb{R},
\end{align*}
for suitable kernels $\theta,\theta_0 :\mathbb{R}\to \mathbb{R}$, the solution to \eqref{SDDE} is a moving average of the type \eqref{MA}. On the other hand, Example \ref{BSS} provides an example of $(Z_t)_{t\in \mathbb{R}}$ where the solution is not of the form \eqref{MA}. Next, in Section~\ref{levelmodelsection}, we briefly discuss existence and uniqueness of solutions to \eqref{levelsCARMA} and provide a few examples. Section~\ref{proofs} contains some technical results together with proofs of all the presented results.

Our proofs rely heavily on the theory of Fourier (and, more generally, bilateral Laplace) transforms, in particular it concerns functions belonging to certain Hardy spaces (or to slight modifications of such). Specific types of Musielak-Orlicz spaces will also play an important role in order to show the results.
%We will briefly present the structure of this paper. In Section~\ref{generalSection} we give an existence results for \eqref{IncCARMA}. Section~\ref{SDDEsection} is concerned with the SDDE in \eqref{SDDE}, where we present an existence result, along with a uniqueness result, and we discuss the structure of the solution more detailed. Moreover, we dedicate a part of this section to an interesting special case (covering its relation to invertible CARMA processes). In Section~\ref{SDDEextend} we depart slightly from the ARMA type structure by replacing the increment moving average term by a more general noise. Section~\ref{levelmodelsection} focuses on the specification \eqref{levelsCARMA} and we elaborate on the two examples of the level autoregressive term presented above. Finally, Section~\ref{proofs} contains some technical results together with proofs of all the presented results.

\medskip

\noindent \textbf{Definitions and conventions:} For $p \in (0,\infty]$ and a (non-negative) measure $\mu$ on the Borel $\sigma$-field $\mathscr{B}(\mathbb{R})$ on $\mathbb{R}$ we denote by $L^p(\mu)$ the usual $L^p$ space relative to $\mu$. If $\mu$ is the Lebesgue measure, we will suppress the dependence on the measure and write $f \in L^p$.
By a finite signed measure we refer to a set function $\mu :\mathscr{B}(\mathbb{R})\to \mathbb{R}$ of the form $\mu = \mu^+ - \mu^-$, where $\mu^+$ and $\mu^-$ are two measures which are mutually singular. Integration of a function $f:\mathbb{R}\to \mathbb{R}$ is defined in an obvious way whenever $f \in L^1(\vert \mu \vert)$, where $\vert \mu \vert := \mu^+ + \mu^-$. For any given finite signed measure $\mu$ set 
\begin{align*}
D(\mu)= \biggr\{z \in \mathbb{C}\; :\; \int_\mathbb{R}e^{\Real (z)u}\, \vert \mu\vert (du) < \infty\biggr\}.
\end{align*}
Then we define the bilateral Laplace transform $\LL : D(\mu) \to \CC$ of $\mu$ by 
\begin{align*}
\mathcal{L}[\mu](z) = \int_\RR e^{z u}\, \mu(du)
\end{align*}
and the Fourier transform by $\mathcal{F}[\mu](y) = \mathcal{L}[\mu](iy)$ for $y \in \mathbb{R}$. If $f\in L^1$ we will write $\mathcal{L}[f] = \mathcal{L}[f(u)\, du]$. We note that $\mathcal{F}[f]\in L^2$ when $f\in L^1 \cap L^2$ and that $\mathcal{F}$ can be extended to an isometric isomorphism from $L^2$ onto $L^2$ by Plancherel's theorem. 

For two finite signed measures $\mu$ and $\nu$ we define the convolution $\mu \ast \nu$ as 
\begin{align*}
\mu \ast \nu (B) = \int_\mathbb{R}\int_\mathbb{R}  \mathds{1}_B (u+v)\mu(du)\nu(dv)
\end{align*}
for any Borel set $B$. For $f,g:\mathbb{R}\to \mathbb{R}$, $g$ being right-continuous, of locally bounded variation, and with Lebesgue-Stieltjes measure $g(du)$, we define $f\ast g$ as  
\begin{align}\label{functionConv}
f \ast g (t) = \int_\mathbb{R} f(t-u)\, g (du)
\end{align}
if $t\in \mathbb{R}$ is such that $f(t-\cdot) \in L^1(\vert g(du)\vert)$ and $f\ast g (t) = 0$ otherwise. (Note that this definition should not be confused with the standard convolution between two functions.) Recall that a process $(L_t)_{t\in \mathbb{R}}$, $L_0 = 0$, is called a (two-sided) L\'{e}vy process if it has stationary and independent increments and càdlàg sample paths (for details, see \cite{Sato}). Let $(L_t)_{t\in \mathbb{R}}$ be a centered L\'{e}vy process with Gaussian component $\sigma^2$ and L\'{e}vy measure $\nu$. Then, for any measurable function $f:\mathbb{R} \to \mathbb{R}$ satisfying
\begin{align}\label{MusielakCond}
\int_\mathbb{R}\biggr(f(u)^2 \sigma^2 + \int_\mathbb{R}( xf(u) )^2 \wedge \vert xf(u) \vert \nu(dx) \biggr)du < \infty,
\end{align}
the integral of $f$ with respect to $(L_t)_{t\in \mathbb{R}}$ is well-defined and belongs to $L^1(\mathbb{P})$ (see \cite[Theorem 3.3]{Rosinski_spec}).

\section{The SDDE setup}\label{SDDEsection}

%We start this section by discussing the existence and uniqueness of solutions to an SDDEs of the form \eqref{SDDE}. Next, we consider some concrete examples of delay measures $\eta$ and noise $$. In particular, we discuss Ornstein-Uhlenbeck processes, the relationship between SDDEs and CARMA processes, and we consider a fractional Lévy process as noise which allows us to gain long memory properties of the solution process. 

%\subsection{Existence and uniqueness}\label{existenceanduniqueness}

%Suppose $\phi(du) =f(u)du$, where $f:\mathbb{R}\to \mathbb{R}$ induces a finite Lebesgue-Stieltjes measure, that is, $f(t)-f(s) = \eta ((s,t])$ for $s<t$ and some finite measure $\eta$. Then 
%\begin{align*}
%\phi_{s,t}(du) = [f(t-u)-f(s-u)]du = \int_{[0,\infty)}\mathds{1}_{(s,t]}(u+v)\, \eta(dv) \, du
%\end{align*}
%and consequently, (\ref{generalModel}) can be written as
Recall that, for a given finite signed measure $\eta$ on $[0,\infty)$ and a measurable process $(Z_t)_{t\in \mathbb{R}}$ with stationary increments and $\mathbb{E}[\vert Z_t\vert]<\infty$ for all $t$, we are interested in the existence and uniqueness of a stationary process $(X_t)_{t\in \mathbb{R}}$ with $\mathbb{E}[\vert X_0\vert]<\infty$ which satisfies
\begin{align}\label{SDDEmodel}
X_t - X_s = \int_s^t \int_{[0,\infty)} X_{u-v} \, \eta(dv) \, du +Z_t - Z_s
\end{align}
almost surely for each $s<t$. In line with \cite{GK} we will construct the solution as a convolution of $(Z_t)_{t\in \mathbb{R}}$ and a deterministic kernel $x_0:\mathbb{R}\to \mathbb{R}$ characterized through $\eta$. Lemma~\ref{autoregKernel} and Corollary~\ref{cadlagBounded} introduce $x_0$ and provide some of its properties. In the formulation we 
% The equation (\ref{SDDEmodel}) is known as a stochastic delay differential equation (SDDE), and existence and (distributional) uniqueness results have been found when $\eta$ is compactly supported for different noises, see e.g. \cite{GK,KuMen,Mohammed,ReissRiedle}. 
%Since (\ref{SDDEmodel}) is a subclass of (\ref{generalModel}) we have, in light of Theorem~\ref{generalExistence}, already found a solution to the SDDE. In this section, however, we will impose some moment assumptions on $\eta$, which in turn will imply that we can find solutions to any valid choice of $\theta$ (and even to noises beyond the increment moving averages, cf. Section~\ref{SDDEextend}). 
will say that $\eta$ has $n$-th moment, $n \in \mathbb{N}$, if $v\mapsto v^n \in L^1(\vert \eta \vert)$ and that $\eta$ has an exponential moment of order $\delta>0$ if $v \mapsto e^{\delta v}\in L^1(\vert \eta \vert)$. Finally, we define the function
\begin{align}\label{hFunction}
h(z):=-z - \mathcal{L}[\eta](z),\quad z \in D(\eta).
\end{align}
%In relation to the conditions below we observe that $h(iy)\neq 0$ for every $y \in \mathbb{R}$ implies $\eta ([0,\infty)) \neq 0$, while $h(z) \neq 0$ for all $z \in \CC$ with non-positive real part implies $\eta ([0,\infty))<0$. 
%According to Lemma~\ref{autoregKernel} below, this ensures the existence of an autoregressive kernel function. Observe that we have the relation $\phi_{0,1}(\mathbb{R}) =  \eta (\mathbb{R})$, thus $\eta (\mathbb{R})=0$ if $\phi$ is finite. Therefore, by assuming $h(iy) \neq 0$ for all $y \in \mathbb{R}$, we have in particular $\eta(\RR) = - h(0) \neq 0$ and $\vert \phi \vert (\mathbb{R})=\infty$. We will deal with the case where $\vert \phi \vert (\mathbb{R})<\infty$ in Section~\ref{levelmodelsection}. 

\begin{lemma}\label{autoregKernel}
Suppose that $h(iy) \neq 0$ for all $y \in \mathbb{R}$. Then there exists a unique function $x_0: \mathbb{R}\to \mathbb{R}$, referred to as the autoregressive kernel, which meets $u \mapsto x_0 (u)e^{cu}\in L^2$ for all $c \in [a,0]$ and a suitably chosen $a<0$, and which satisfies
\begin{align}\label{x0rel}
x_0(t) = \mathds{1}_{[0,\infty)}(t) + \int_{-\infty}^t \int_{[0,\infty)} x_0(u-v) \, \eta (dv) \, du
\end{align}
for all $t\in \mathbb{R}$. Furthermore, $x_0$ is characterized by $\LL [x_0](z) = 1/h(z)$ for $z \in \CC$ with real part in $(a,0)$, and the following statements hold:
 \begin{enumerate}[(i)]
 \item If $\eta$ has $n$-th moment for some $n \in \mathbb{N}$, then $x_0 \in L^q$ for all $q \in [1/n,\infty]$ and, in particular,
 \begin{align}\label{minusOne}
 \int_\mathbb{R}\int_{[0,\infty)} x_0(u-v)\, \eta (dv)\, du = -1.
 \end{align}
 \item If $\eta$ has an exponential moment of order $\delta>0$, then there exists $\varepsilon \in (0,\delta]$ such that $u \mapsto x_0 (u)e^{cu}\in L^2$ for all $c \in [a,\varepsilon]$ and, in particular, $x_0 \in L^q$ for all $q \in (0,\infty]$.
 \item If $h(z) \neq 0$ for all $z \in \CC$ with non-positive real part, then $x_0 (t) = 0$ for all $t<0$.
 \end{enumerate}
\end{lemma}
By \eqref{x0rel} it follows that $x_0$ induces a Lebesgue-Stieltjes measure $x_0(du)$. From Lemma~\ref{autoregKernel} we deduce immediately the following properties of $x_0(du)$:
 
\begin{corollary}\label{cadlagBounded} Suppose that $h(iy) \neq 0$ for all $y \in \mathbb{R}$. Then the autoregressive kernel defines a Lebesgue-Stieltjes measure, and it is given by
\begin{align*}
x_0 (du) = \delta_0 (du) + \Bigr( \int_{[0,\infty)}x_0(u-v) \, \eta (dv)\Bigr) du.
\end{align*}
A function $\theta:\mathbb{R}\to \mathbb{R}$ is integrable with respect to $x_0(du)$ if and only if $\theta \in L^1(\mu)$, where $\mu$ is the (non-negative) measure given by
\begin{align}\label{muMeasure}
\mu (du) = \int_{[0,\infty)} \vert x_0(u-v) \vert  \, \vert \eta \vert (dv) \, du.
\end{align}
In particular, if $\eta$ has $n$-th moment for some $n \in \mathbb{N}$, then $x_0(du)$ has moments up to order $n-1$.
\end{corollary}
%\begin{remark}\label{diffRemark} Since there is only one integrable function that meets (\ref{x0rel}) for all $t \in \mathbb{R}$, one may obtain the autoregressive kernel by solving this equation. This strategy will e.g. be used in Example~\ref{OUexample}.
%\end{remark}
\begin{remark}
Consider the case where $\eta$ is supported on $[0,r]$ for some $r>0$. Then it is known from the theory of deterministic delay equations (see \cite{delay,GK}) that there exists one and only one function $y:\mathbb{R} \to \mathbb{R}$ with
\begin{align}\label{homogen}
y(t) = \begin{cases} \begin{array}{lcl} 1 + \int_0^t \int_{[0,r]} y(u-v) \, \eta (dv)\, du & \text{if} & t \geq 0 \\
0 & \text{if} & t <0.
\end{array}
\end{cases}
\end{align}
The equation \eqref{homogen} is called the homogeneous equation and $y$ is the associated fundamental solution. Consequently, in light of Lemma~\ref{autoregKernel}, the autoregressive kernel $x_0$ coincides with the fundamental solution $y$ if $h(z) \neq 0$ for all $z \in \CC$ with non-positive real part.

In the literature the equation $h(z)=0$ is called the characteristic equation (associated to (\ref{SDDEmodel})). For a compactly supported $\eta$ it is known that $h$ will have infinitely many roots except in the Ornstein-Uhlenbeck case where $\eta = -\lambda \delta_0$ for some $\lambda>0$. This will in turn result in a fundamental solution which does not have a closed form expression. For more on SDDEs with a compactly supported $\eta$, see e.g. \cite{GK,KuMen}.
% In Section~\ref{stateSpace} we will see that it is indeed possible that $h$ only has a finite number of roots when $\eta$ has unbounded support, and this allows us to obtain explicit expressions for $x_0$. In particular, it is a key property when it comes to matching the kernel functions of the CARMA processes.
\end{remark}
%
%\begin{theorem}\label{x0PureLevy} Suppose that $h(iy)\neq 0$ for all $y \in \mathbb{R}$. Then there exists a function $x_0 :\mathbb{R}\to \mathbb{R}$ in $L^2$, referred to as the autoregressive kernel, and a suitable $a<0$ such that
%\begin{align*}
%\mathcal{L}[x_0](z) = \frac{1}{h(z)}
%\end{align*}
%for $z \in \mathcal{S}_{a,0}$. Moreover, it holds that $X_t = x_0 \ast L (t)$, $t \in \mathbb{R}$, satisfies (\ref{SDDEmodel}) with $\theta = \mathds{1}_{[0,\infty)}$ in any of the following two cases:
%\begin{enumerate}[(i)]
%\item If $(L_t)_{t\in \mathbb{R}}$ is centered and square integrable.
%\item If $(L_t)_{t\in \mathbb{R}}$ is integrable and $\eta$ has first moment.
%\end{enumerate}
%Moreover, if $h(z) \neq 0$ for all $z \in \overline{\mathcal{S}_{-\infty,0}}$, $x_0(t) = 0$ for $t<0$ and in particular, $(X_t)_{t\in \mathbb{R}}$ is causal.
%\end{theorem}
With the autoregressive kernel in hand we present our main result of this section:
\begin{theorem}\label{psiCharacterization} Suppose that $\eta$ is a finite signed measure with second moment and $h(iy) \neq 0$ for all $y \in \mathbb{R}$. Then the process
\begin{align}\label{convSolution}
X_t =  Z_t + \int_\mathbb{R}Z_{t-u} \int_{[0,\infty)} x_0 (u-v) \, \eta (dv) \, du,\quad t \in \mathbb{R},
\end{align}
is well-defined and it is the unique integrable stationary solution (up to modification) to \eqref{SDDEmodel}. 
%If $h(z)\neq 0$ for all $z \in \CC$ with non-positive real part, then $ x_0(t) =0$ for $t<0$ and in particular, $(X_t)_{t\in \mathbb{R}}$ is causal.
\end{theorem}
Often, $(Z_t)_{t\in \mathbb{R}}$ is given by
\begin{align}\label{simmaZ}
Z_t = \int_\mathbb{R}\big[\theta (t-u)- \theta_0 (-u) \big]\, dL_u,\quad t \in \mathbb{R},
\end{align}
for some centered L\'{e}vy process $(L_t)_{t\in \mathbb{R}}$ and measurable functions $\theta,\theta_0:\mathbb{R}\to \mathbb{R}$ such that $u\mapsto \theta (t+u)-\theta_0(u)$ satisfies \eqref{MusielakCond} for $t>0$. The next result shows that the (unique) solution to \eqref{SDDEmodel} is a L\'{e}vy-driven moving average in this particular setup. In the formulation we recall the notation $f\ast g (t) = \int_\mathbb{R} f(t-u)\, g(du)$ for sufficiently nice functions $f,g:\mathbb{R}\to \mathbb{R}$ (cf. \eqref{functionConv}).
\begin{corollary}\label{simmaNoise}
Let the setup be as in Theorem~\ref{psiCharacterization} and suppose that $(Z_t)_{t\in \mathbb{R}}$ is of the form \eqref{simmaZ}. Then the unique integrable stationary solution to \eqref{SDDEmodel} is given by
\begin{align*}
X_t = \int_\mathbb{R} \theta \ast x_0 (t-u)\, dL_u,\quad t \in \mathbb{R}.
\end{align*} 
In particular if $Z_t = L_t$ for $t\in \mathbb{R}$, we have that
\begin{align*}
X_t = \int_\mathbb{R}x_0 (t-u)\, dL_u,\quad t \in \mathbb{R}.
\end{align*}
\end{corollary}
%\item If $(Z_t)_{t\in \mathbb{R}}$ is centered and square integrable Lévy process, respectively if $(Z_t)_{t\in \mathbb{R}}$ is %integrable Lévy process and $\eta$ has first moment, then 
%\begin{align*}
%X_t = \int_\RR x_0(t-u) \, dZ_u, \quad t \in \mathbb{R},
%\end{align*}
%is centered and square integrable, respectively integrable, and the unique stationary solution (up to modification)  to %\eqref{SDDEmodel}.
%\end{enumerate}

\begin{remark}\label{zeroH}
The assumption $h(0)=-\eta ([0,\infty))\neq 0$ is rather crucial in order to find stationary solutions and is the analogue of assuming that the AR polynomial in a discrete-time ARMA setting does not have a root at $1$ or, in other words, the AR coefficients do not sum to zero. For instance, the setup where $\eta\equiv 0$ will satisfy $h(iy)\neq 0$ for all $y\in \mathbb{R}\setminus \{0\}$, but if $(Z_t)_{t\in \mathbb{R}}$ is a L\'{e}vy process, the SDDE \eqref{SDDEmodel} cannot have stationary solutions. In Example~\ref{SDDEstatInc}, we show how one can find solutions with stationary increments for a reasonably large class of delay measures $\eta$ with $\eta ([0,\infty)) = 0$.
\end{remark}

\begin{remark}
It should be stressed that for more restrictive choices of $\eta$, and in case $(Z_t)_{t\in \mathbb{R}}$ is a L\'{e}vy process, solutions sometimes exist even when $\mathbb{E}[\vert Z_1\vert] = \infty$. Indeed, if $\eta$ is compactly supported and $h$ has no roots with non-positive real part, one only needs that $\mathbb{E}[\log^+ \vert Z_1 \vert]< \infty$ to ensure that a stationary solution exists. We refer to \cite{GK,ReissRiedle} for details.
\end{remark}

We now present some concrete examples of SDDEs. The first three examples concern the specification of the delay measure and the last two concern the specification of the noise. 

\begin{example}\label{OUexample} Let $\lambda \neq 0$ and consider the equation
\begin{align}\label{theQOU}
X_t - X_s = -\lambda \int_s^t X_u du +Z_t - Z_s,\quad s<t.
\end{align}
In the setup of \eqref{SDDEmodel} this corresponds to $\eta = -\lambda \delta_0$. With $h$ given by \eqref{hFunction}, we have $h(z) =\lambda -z\neq 0$ for every $z\in \mathbb{C}$ with real part not equal to $\lambda$, and hence Theorem~\ref{psiCharacterization} implies that there exists a stationary process $(X_t)_{t\in \mathbb{R}}$ with $\mathbb{E}[\vert X_0\vert]<\infty$ satisfying \eqref{theQOU}. According to Lemma~\ref{autoregKernel} the autoregressive kernel function $x_0$ can be determined through its Laplace transform on the complex numbers with real part in $(a,0)$ for a suitable $a<0$ as
\begin{align*}
\mathcal{L}[x_0](z) = \frac{1}{\lambda-z} = \begin{cases} \begin{array}{lcl}\mathcal{L}\big[\mathds{1}_{[0,\infty)}e^{-\lambda \cdot} \big](z) & \text{if} & \lambda >0\\
\mathcal{L}\big[-\mathds{1}_{(-\infty,0)}e^{-\lambda \cdot} \big](z) & \text{if} & \lambda <0.
\end{array}
\end{cases}
\end{align*}
Consequently, by Theorem~\ref{psiCharacterization},
\begin{align}\label{zRepresentation}
X_t = \begin{cases}
\begin{array}{lcl}
Z_t - \lambda e^{-\lambda t}\int_{-\infty}^t Z_u e^{\lambda u}\, du  &\text{if}& \lambda >0\\
Z_t + \lambda e^{-\lambda t}\int_t^\infty Z_u e^{\lambda u}\, du &\text{if}& \lambda <0
\end{array}
\end{cases}
\end{align}
for $t\in \mathbb{R}$. Ornstein-Uhlenbeck processes satisfying \eqref{theQOU} have already been studied in the literature, and representations of the stationary solution have been given, see e.g. \cite[Theorem~2.1, Proposition~4.2]{QOU}. 
\end{example}

\begin{example}\label{CARMAexample}%[CARMA processes]

Let $(L_t)_{t\in \mathbb{R}}$ be a L\'{e}vy process with $\mathbb{E}[\vert L_1\vert]<\infty$. Recall that $(X_t)_{t \in \RR}$ is said to be a CARMA($p,q$) process, $p>q$, if 
\begin{align*}
X_t = \int_{-\infty}^t g(t-u)\, dL_u,\quad t \in \mathbb{R},
\end{align*}
where the kernel $g$ is characterized by 
\begin{align*}
\FF[g](y) = \frac{Q(-iy)}{P(-iy)},\quad y \in \mathbb{R},
\end{align*}
for two polynomials $P$ and $Q$ of order $p$ and $q$ ($p>q$), respectively, and where $P$ has no roots with non-negative real part. Consider the special case of a CARMA($2,1$) process and write 
\begin{align*}
P(z) = z^2 + a_1 z+ a_2 \quad \text{and} \quad Q(z) = b_0 + z.
\end{align*}
Assume that the invertibility assumption $b_0 >0$ holds (note that this assumption corresponds to assuming that $Q(z) \neq 0$ for all $z\in \mathbb{C}$ with non-negative real part). In this case
\begin{align*}
\FF[g](y) = \frac{Q(-iy)}{P(-iy)} = \frac{1}{ -iy + a_1 - b_0 + \frac{a_2 - b_0(a_1-b_0)}{b_0-iy} }
\end{align*}
for $y\in \mathbb{R}$. By choosing the delay measure $\eta(dv) = (b_0-a_1) \delta (dv) - (a_2 - b_0(a_1-b_0))e^{-b_0v} dv$ it is seen that the function $h$ given in \eqref{hFunction} satisfies $1/h = \FF[g]$, and therefore it follows by Theorem~\ref{psiCharacterization} that the CARMA($2,1$) process is the unique solution to the SDDE with delay measure $\eta$. In fact, any CARMA($p,q$) process satisfying $Q(z) \neq 0$ for all $z\in \mathbb{C}$ with non-negative real part can be represented as the solution to an equation of the SDDE type (see \cite[Theorem~4.8]{MSDDE} for a precise statement).
\end{example}

\begin{example}\label{GammaExample}

In this example we consider a delay measure $\eta$ where the corresponding solution to the SDDE in \eqref{SDDEmodel} may be regarded as a CARMA process with fractional polynomials. Specifically, consider
\begin{align*}
\eta (dv) = \alpha_1 \delta_0 (dv) + \frac{\alpha_2}{\Gamma (\beta)} \mathds{1}_{[0,\infty)}(v) v^{\beta-1} e^{-\gamma v} \, dv
\end{align*}
where $\beta, \gamma >0$ and $\Gamma$ is the gamma function. By Lemma \ref{autoregKernel} and Theorem \ref{psiCharacterization} the associated SDDE has a unique solution with kernel $x_0$ satisfying $x_0(t) = 0$ for $t<0$ if
\begin{align}\label{hGamma}
h(z) = -z - \alpha_1 - \alpha_2 (\gamma - z)^{-\beta}
\end{align}
is non-zero for all $z\in \CC$ with non-positive real part. We see that $h$ is a ratio of fractional polynomials, and if $\beta$ is an integer it is a ratio of two regular polynomials as is the case for CARMA processes. We now consider sufficient conditions for \eqref{hGamma} to be satisfied. Assume $\alpha_1 <0$. Then we find
\begin{align*}
\vert  -z - \alpha_1 - \alpha_2 (\gamma - z)^{-\beta} \vert &\geq - \alpha_1  - \vert \alpha_2 \vert \vert (\gamma - z)^{-\beta} \vert \\
& \geq  - \alpha_1  - \vert \alpha_2 \vert  \gamma^{-\beta}
\end{align*}
for all $z\in \CC$ with non-positive real part. Consequently, a sufficient assumption for \eqref{hGamma} to hold is $\alpha_1 + \vert \alpha_2 \vert \gamma^{-\beta} <0$. If we assume $\alpha_1,\alpha_2 <0$ and $\beta < 1$ then it is clear that \eqref{hGamma} holds since $(\gamma -z)^{-\beta}$ has positive real part for all $z\in \CC$ with non-positive real part. We conclude that there exists a unique stationary solution whenever \textit{i)} $\alpha_1 <0$ and $\alpha_1 + \vert \alpha_2 \vert \gamma^{-\beta} <0$, or \textit{ii)} $\alpha_1, \alpha_2 <0$ and $\beta<1$.
\end{example}

\begin{example}%[Fractional Lévy process as noise]

Let $\eta$ be any finite signed measure with second moment which satisfies $h(iy) \neq 0$ for all $y\in \mathbb{R}$. Consider the case where $(Z_t)_{t \in \RR}$ is a fractional Lévy process, that is,
\begin{align*}
Z_t = \frac{1}{\Gamma (1+d)}\int_\mathbb{R} \big[(t-u)_+^d - (-u)_+^d\big] \, dL_u,\quad t \in \mathbb{R},
\end{align*}
where $d \in (0,1/2)$ and $(L_t)_{t \in \RR}$ is a centered and square integrable Lévy process. Let 
\begin{align*}
\theta(t) = \frac{t_+^d}{\Gamma(1+d)}.
\end{align*}
Then it follows by Corollary~\ref{simmaNoise} that the solution to \eqref{SDDEmodel} takes the form
\begin{align*}
 X_t = \int_\mathbb{R} \theta \ast x_0(t-u) \, dL_u, \quad t \in \mathbb{R}.
\end{align*}	
It is not too difficult to show that $\theta \ast x_0$ coincides with the left-sided Riemann-Liouville fractional integral of $x_0$, and hence $X_t = \int_\mathbb{R} x_0 (t-u)\, dZ_u$, where the integral with respect to $(Z_t)_{ t \in \RR}$ is defined as in \cite{Marquardt}. Consequently, we can use the proof of \cite[Theorem 6.3]{Marquardt} to deduce that $(X_t)_{t \in \RR}$ has long memory in the sense that its autocovariance function is hyperbolically decaying at $\infty$: 
\begin{align}\label{longMemory}
\gamma_X (t) := \mathbb{E}[X_tX_0] \sim \frac{\Gamma (1-2d)}{\Gamma (d) \Gamma (1-d)} \frac{\mathbb{E}[L_1^2]}{h(0)^2} t^{2d-1}, \quad t \to \infty.
\end{align}
In particular, \eqref{longMemory} shows that $\gamma_X \notin L^1$.
\end{example}

The last example presents a situation where Theorem~\ref{psiCharacterization} is applicable, but $(Z_t)_{t\in \mathbb{R}}$ is not of the form \eqref{simmaZ}. It is closely related to \cite[Corollary~2.3]{QOU}.
\begin{example} \label{BSS}
Consider a filtration $\mathscr{F} = (\mathcal{F}_t)_{t\in \mathbb{R}}$ and let $(B_t)_{t\in \mathbb{R}}$ be an $\mathscr{F}$-Brownian motion. Moreover, let $(\sigma_t)_{t\in \mathbb{R}}$ be a predictable process with $\sigma_0 \in L^2(\mathbb{P})$. Finally, assume that $(\sigma_t,B_t)_{t\in \mathbb{R}}$ and $(\sigma_{t+u},B_{t+u}-B_u)_{t \in \mathbb{R}}$ have the same finite-dimensional marginal distributions for all $u \in \mathbb{R}$. In this case
\begin{align*}
Z_t = \int_0^t \sigma_s dB_s, \quad t \in \mathbb{R},
\end{align*}
is well-defined, continuous and square integrable, and it has stationary increments. (Here we use the convention $\int_0^t := -\int_t^0$ when $t<0$.) Under the assumptions that $h(z) \neq 0$ for all $z \in \mathbb{C}$ with non-positive real part and $\eta$ has second moment, Theorem~\ref{psiCharacterization} implies that there exists a unique stationary solution $(X_t)_{t\in\mathbb{R}}$ to \eqref{SDDEmodel} and, since $x_0(t) = 0$ for $t <0$, it is given by
\begin{align*}
X_t &= Z_t + \int_0^\infty Z_{t-s} \int_{[0,\infty)} x_0(s-v)\, \eta (dv)\, ds\\
 &= -\int_0^\infty \int_{t-s}^t \sigma_u\, dB_u \int_{[0,\infty)} x_0(s-v)\, \eta (dv)\, ds\\
&= -\int_{-\infty}^t \sigma_u\int_{t-u}^\infty \int_{[0,\infty)} x_0(s-v)\, \eta (dv)\, ds\,  dB_u\\
&= \int_{-\infty}^t  x_0(t-u)\sigma_u\, dB_u
\end{align*}
for $t\in \mathbb{R}$, where we have used Corollary~\ref{cadlagBounded} and an extension of the stochastic Fubini given in \cite[Chapter~IV, Theorem~65]{Protter} to integrals over unbounded intervals.
\end{example}

\section{The level model}\label{levelmodelsection}
In this section we consider the equation 
\begin{align}\label{levelModel}
\begin{aligned}
X_t  = \int_0^\infty X_{t-u}\,  \phi (du) + \int_{-\infty}^t\theta(t-u)\, dL_u, 
\end{aligned}
\end{align}
where $\phi$ is a finite signed measure on $(0,\infty)$, $(L_t)_{t\in \mathbb{R}}$ is a centered L\'{e}vy process and $\theta:\mathbb{R}\to \mathbb{R}$ is a measurable function, which is vanishing on $(-\infty,0)$ and satisfies \eqref{MusielakCond}. 
%As addressed in Section~\ref{SDDEsection}, if $\phi (du) = f(u)du$, where $f\in L^1$ induces a finite Lebesgue-Stieltjes measure $\eta$, then we necessarily have that $\eta (\mathbb{R})=0$ and $h(0) = 0$. Consequently, this situation was not handled in the above section.
Equation \eqref{levelModel} can be solved using the backward recursion method under the contraction assumption $\vert \phi \vert ((0,\infty))<1$, and this is how we obtain Theorem~\ref{levelResult}. For the noise term we will put the additional assumption that $\mathbb{E}[L_1^2]<\infty$, and hence (in view of \eqref{MusielakCond}) that $\theta\in L^2$. In the formulation we will denote by $\phi^{\ast n}$ the $n$-fold convolution of $\phi$, that is, $\phi^{\ast n} :=\phi\ast \cdots \ast \phi$ for $n \in \mathbb{N}$ and $\phi^{\ast 0}= \delta_0$.
 
\begin{theorem}\label{levelResult} Let $(L_t)_{t\in \mathbb{R}}$ be a L\'{e}vy process with $\mathbb{E}[L_1] = 0$ and $\mathbb{E}[L_1^2]<\infty$, and suppose that $\theta \in L^2$, and $\vert \phi\vert ((0,\infty))<1$. Then there exists a unique square integrable solution to \eqref{levelModel}. It is given by 
\begin{align*}
X_t = \int_{-\infty}^t \psi(t-u)\, dL_u, \quad t \in \RR,
\end{align*}
where $\psi := \sum_{n=0}^\infty\theta \ast  \phi^{\ast n}$ exists as a limit in $L^2$ and is vanishing on $(-\infty,0)$. %Consequently, $(X_t)_{t\in \mathbb{R}}$ is causal.
\end{theorem}

\begin{remark}\label{generalExistenceRemark}
One can naturally argue for the existence of solutions to \eqref{levelModel} beyond Theorem~\ref{levelResult}. In particular, if a function $\psi\in L^2$ satisfies
\begin{align}\label{levelKernelEqn}
\psi (t-r) = \int_0^\infty \psi (t-u-r) \phi (du) + \theta (t-r)
\end{align}
for Lebesgue almost all $r \in \mathbb{R}$ for each fixed $t\in \mathbb{R}$, then by a stochastic Fubini result the moving average $X_t = \int_\mathbb{R} \psi (t-u)\, dL_u$, $t\in \mathbb{R}$, is a solution to \eqref{levelModel}. The deterministic equation \eqref{levelKernelEqn} has a solution if
\begin{align}\label{supCond}
\sup_{a<x <0} \int_\mathbb{R} \Big\vert \frac{\LL[\theta](x+iy)}{1-\LL[\phi](x+iy)} \Big\vert^2 dy <\infty
\end{align}
for some $a<0$, and in this case the solution $\psi$ is characterized by 
\begin{align}\label{psiLevel}
\LL[\psi](z) = \frac{\LL[\theta](z)}{1-\LL[\phi](z)}
\end{align}
for $z \in \CC$ with real part in $(a,0)$. This is a consequence of Lemma~\ref{helpLemma} given in the next section. Note that the condition in \eqref{supCond} is satisfied under the assumptions of Theorem~\ref{levelResult} and that \eqref{psiLevel} is an alternative characterization of the function $\sum_{n=0}^\infty\theta \ast  \phi^{\ast n}$ in this case.
\end{remark}

\begin{example}\label{expMeasure}
Suppose that $(L_t)_{t\in \mathbb{R}}$ is a L\'{e}vy process with $\mathbb{E}[L_1] = 0$ and $\mathbb{E}[L_1^2]<\infty$, and let $\theta \in L^2$. For $\alpha \in \mathbb{R}$ and $\beta >0$, consider 
\begin{equation*}
\phi (du) = \alpha e^{-\beta u}\mathds{1}_{[0,\infty)}(u) \, du
\end{equation*}
and define the measure
\begin{equation*}
\xi(du) = e^{\alpha u}\phi (du) = \alpha e^{(\alpha-\beta)u}\mathds{1}_{[0,\infty)}(u)\, du.
\end{equation*}
We start by assuming $-1< \alpha/\beta <1$. Then Theorem~\ref{levelResult} implies that the solution kernel $\psi$ is given by 
\begin{align*}
\psi = \sum_{n=0}^\infty \theta \ast \phi^{\ast n} = \theta + \theta \ast \sum_{n=1}^\infty \phi^{\ast n} =  \theta + \theta \ast \xi.
\end{align*} 
We can use Remark~\ref{generalExistenceRemark} to find the solution kernel when $ \alpha/\beta \leq -1$. In particular, we find that 
\begin{align*}
\frac{\LL[\theta](z)}{1-\LL[\phi](z)} = \frac{\LL[\theta](z)}{1 - \alpha \frac{1}{\beta-z}} = \LL[\theta](z) + \LL[\theta](z) \frac{\alpha}{\beta-\alpha-z} = \LL[ \theta + \theta \ast \xi](z),
\end{align*}
and therefore we still have $\psi = \theta + \theta \ast \xi$ is the solution kernel. We conclude that a solution to \eqref{levelModel} is given by 
\begin{align*}
X_t = \int_{-\infty}^t\theta(t-u)\, dL_u +  \int_{-\infty}^t\theta \ast \xi(t-u)\, dL_u,\quad t \in \mathbb{R},
\end{align*}
when $\alpha/\beta<1$.   

\end{example}
The next example relates \eqref{levelModel} to \eqref{SDDEmodel} in a certain setup.
\begin{example}\label{SDDEstatInc}
We will give an example of an SDDE where Theorem~\ref{psiCharacterization} does not provide a solution, but where a solution can be found by considering an associated level model. Consider the SDDE model \eqref{SDDEmodel} in the case where $\eta$ is absolutely continuous and its cumulative distribution function $F_\eta (t) := \eta ([0,t])$, $t\geq 0$, satisfies
\begin{align}\label{fEtaAssump}
\int_0^\infty \vert F_\eta (t)\vert\, dt <1.
\end{align}
Under this assumption $h$ defined in \eqref{hFunction} satisfies $h(0) = -\eta ([0,\infty)) = 0$, and hence Theorem~\ref{psiCharacterization} does not apply (cf.\ Remark~\ref{zeroH}). In fact, using Fubini's theorem and integration by parts on the delay term, the equation may be written as
\begin{align}\label{SDDEibp}
X_t - X_s = \int_0^\infty \big[X_{t-u}-X_{s-u} \big] F_\eta (u)\, du + Z_t -Z_s,\quad s<t.
\end{align}
This shows that uniqueness does not hold, since if $(X_t)_{t\in \mathbb{R}}$ is a solution then so is $(X_t + \xi)_{t\in \mathbb{R}}$ for any $\xi \in L^1(\mathbb{P})$. Moreover, as noted in Remark~\ref{zeroH}, we cannot expect to find stationary solutions in this setup. In the following let us restrict the attention to the case where 
\begin{align*}
Z_t = \int_\mathbb{R} [f (t-u) - f_0 (-u) ]\, dL_u, \quad t \in \mathbb{R},
\end{align*}
for a given L\'{e}vy process $(L_t)_{t\in \mathbb{R}}$ with $\mathbb{E}[L_1] = 0$ and $\mathbb{E}[L_1^2]<\infty$, and for some functions $f,f_0:\mathbb{R}\to \mathbb{R}$, vanishing on $(-\infty,0)$, such that $u\mapsto f(t+u) - f_0 (u)$ belongs to $L^2$. Using Theorem~\ref{levelResult} we will now argue that there always exists a centered and square integrable solution with stationary increments in this setup and that the increments of any two such solutions are identical. 

To show the uniqueness part, suppose that $(X_t)_{t\in \mathbb{R}}$ is a centered and square integrable stationary increment process which satisfies \eqref{SDDEibp}. Then, for any given $s>0$, we have that the increment process $X(s)_t = X_t-X_{t-s}$, $t\in \mathbb{R}$, is a stationary, centered and square integrable solution to the level equation \eqref{levelModel} with $\phi (du) = F_\eta (u)\, du$ and $\theta  = f - f(\cdot - s)$. By the uniqueness part of Theorem~\ref{levelResult} and \eqref{fEtaAssump} it follows that
\begin{align*}
X(s)_t = \int_\mathbb{R}\psi_s (t-u)\, dL_u,\quad t \in \mathbb{R},
\end{align*}
where $\psi_s (t) = \sum_{n=0}^\infty \int_0^\infty [f(t-u)-f(t-s-u)]\, \phi^{\ast n} (du)$ (the sum being convergent in $L^2$). Consequently, by a stochastic Fubini result, $(X_t)_{t\in \mathbb{R}}$ must take the form
\begin{align}\label{statIncSol}
X_t =\xi + \sum_{n=0}^\infty \int_0^\infty [Z_{t-u}-Z_{-u}]\, \phi^{\ast n} (du),\quad t \in \mathbb{R},
\end{align}
for a suitable $\xi \in L^2(\mathbb{P})$ with $\mathbb{E}[\xi]=0$. Conversely, if one defines $(X_t)_{t\in \mathbb{R}}$ by \eqref{statIncSol} we can use the same reasoning as above to conclude that $(X_t)_{t\in \mathbb{R}}$ is a stationary increment solution to \eqref{SDDEmodel}. It should be stressed that one can find other representations of the solution than \eqref{statIncSol} (e.g., in a similar manner as in Example~\ref{expMeasure}).
\end{example}

A nice property of the model (\ref{levelModel}) is that it can recover the discrete-time ARMA($p$,$q$) process. Example~\ref{discreteCond} gives (well-known) results for ARMA processes by using Remark~\ref{generalExistenceRemark}. For an extensive treatment of ARMA processes, see e.g. \cite{BrockDavis}.

\begin{example}\label{discreteCond}
Let $p,q\in \mathbb{N}_0$ and define the polynomials $\Phi,\Theta:\mathbb{C}\to \mathbb{C}$ by 
\begin{align*}
\Phi (z)= 1-\phi_1z- \cdots - \phi_p z^p \quad \text{and} \quad \Theta (z) = 1+\theta_1z + \cdots +\theta_q z^q,
\end{align*}
where the coefficients are assumed to be real. Let $(L_t)_{t\in \mathbb{R}}$ be a L\'{e}vy process with $\mathbb{E}[L_1] =0$ and $\mathbb{E}[L_1^2]<\infty$, and consider choosing $\phi (du) = \sum_{j=1}^p \phi_j \delta_{j}(du)$ and $\theta (u) = \mathds{1}_{[0,1)}(u) + \sum_{j=1}^q \theta_j \mathds{1}_{[j,j+1)}(u)$. In this case \eqref{levelModel} reads
\begin{align}\label{ARMA}
X_t = \sum_{i=1}^p \phi_i X_{t-i} + Z_t + \sum_{i=1}^q \theta_i Z_{t-i},\quad t \in \mathbb{R},
\end{align}
with $Z_t = L_t - L_{t-1}$. In particular, if $(X_t)_{t\in \mathbb{R}}$ is a solution to (\ref{ARMA}), $(X_t)_{t\in \mathbb{Z}}$ is a usual ARMA process. Suppose that $\Phi (z) \neq 0$ for all $z \in \mathbb{C}$ with $\vert z\vert =1$. Then, by continuity of $\Phi$, there exists $a<0$ such that $1-\mathcal{L}[\phi](z) = \Phi (e^z)$ is strictly separated from $0$ for $z\in \mathbb{C}$ with real part in $(a,0)$. Thus, since $\theta \in L^2$, Remark~\ref{generalExistenceRemark} implies that there exists a stationary solution to (\ref{levelModel}), and it is given by $X_t= \int_\mathbb{R}\psi (t-u)\, dL_u$, $t\in \mathbb{R}$, where $\psi$ is characterized by \eqref{psiLevel}. Choose a small $\varepsilon >0$ and $(\psi_j)_{j \in \mathbb{Z}}$ so that the relation
\begin{align*}
\frac{\Theta (z)}{\Phi (z) } = \sum_{j=-\infty}^\infty \psi_j z^j
\end{align*}
holds true for all $z \in \mathbb{C}$ with $1-\varepsilon < \vert z \vert <1+\varepsilon$. Then 
\begin{align*}
\mathcal{L}[\psi] (z)
= \mathcal{L}[\mathds{1}_{[0,1)}](z)\frac{\Theta (e^{z})}{ \Phi (e^{z})}
=\sum_{j=-\infty}^\infty \psi_j \mathcal{L}[\mathds{1}_{[j,j+1)}](z) 
= \mathcal{L}\biggr[\sum_{j=-\infty}^\infty \psi_j \mathds{1}_{[j,j+1)} \biggr] (z)
\end{align*}
for all $z \in \mathbb{C}$ with a negative real part sufficiently close to zero. Thus, we have the well-known representation $X_t = \sum_{j=-\infty}^{\infty} \psi_j Z_{t-j}$ for $t\in \mathbb{R}$. 
\end{example}

%In view of \eqref{SDDEmodel},

%If we even have that $\boldsymbol \phi (z) \neq 0$ for all $z \in \mathbb{C}$ with $\vert z\vert \leq 1$, Theorem~\ref{levelResult} is applicable with $a= -\infty$, implying that $\psi (u) = 0$ for Lebesgue almost all $u <0$, which in turn gives that $X_t = \sum_{j=0}^\infty \psi_j Z_{t-j}$. Under any of the assumptions on $\boldsymbol \phi$, since $(X_t)_{t\in \mathbb{R}}$ is stationary and satisfies (\ref{levelModel}), $(X_t)_{t \in \mathbb{Z}}$ is stationary and satisfies
%\begin{align*}
%X_t = \sum_{i=1}^p \phi_i X_{t-i} + Z_t + \sum_{i=1}^q \theta_i Z_{t-i}
%\end{align*}
%for $t \in \mathbb{Z}$. In other words, $(X_t)_{t\in \mathbb{Z}}$ is an ARMA($p$,$q$) process.

\section{Proofs and technical results}\label{proofs}

The first result is closely related to the characterization of the so-called Hardy spaces. For more on this topic, see e.g. \cite{Doetsch} or \cite[Section 2.3]{Dym:gaussian}. We will use the notation 
\begin{align*}
\mathcal{S}_{a,b} = \{ z \in \CC \, :\, a< \Real(z) <b \}.
\end{align*}
\begin{lemma}\label{helpLemma}
Let $-\infty \leq a < b\leq \infty$. Suppose that $F:\mathbb{C}\to \mathbb{C}$ is a function which is analytic on the strip $\mathcal{S}_{a,b}$ and satisfies
\begin{align}\label{hardyCond}
\sup_{a<x <b} \int_\mathbb{R} \vert F(x+ iy)\vert^2 dy <\infty.
\end{align}
Then there exists a function $f:\mathbb{R}\to \mathbb{C}$ such that $\mathcal{S}_{a,b}\subseteq D(f)$, $\mathcal{L}[f](z) = F(z)$ for all $z \in \mathcal{S}_{a,b}$, and $u \mapsto f(u)e^{c u} \in L^2$ for all $a\leq c \leq b$.
\end{lemma}

\begin{remark}\label{causalityReason}
If $a=-\infty$, the property $u \mapsto f(u)e^{au} \in L^2$ is understood as $f(u) = 0$ for almost all $u <0$ and similarly, $f(u) =0$ for almost all $u >0$ if $u\mapsto f(u) e^{bu} \in L^2$ for $b=\infty$.
\end{remark}

\begin{proof}[Proof of Lemma~\ref{helpLemma}] Fix $c_1,c_2 \in (a,b)$ with $c_1<c_2$ and $u \in \mathbb{R}$. For any $y >0$, consider (anti-clockwise) integration of $F$ along a rectangular contour $R_y$ with vertices $c_1-i y$, $c_2 - iy$, $c_2 + iy$, and $c_1 + iy$:
\begin{align}
\begin{aligned}\label{contourInteg}
0 &= \oint_{R_y} e^{-zu}F(z)dz \\
&= \int_{c_1}^{c_2} e^{-(x-iy)u}F(x-iy)dx + ie^{-c_2u}\int_{-y}^ye^{-ixu}F(c_2+ix)dx \\
&- \int_{c_1}^{c_2} e^{-(x+iy)u}F(x+iy)dx - i e^{-c_1u}\int_{-y}^ye^{-ixu}F(c_1+ix)dx.
\end{aligned}
\end{align}
For $k=1,2$ it holds that
\begin{align}\label{L2conv}
\biggr(u \mapsto \int_{-y}^y e^{-ixu}F(c_k + ix)dx\biggr) \to \biggr(u\mapsto 2\pi \mathcal{F}^{-1}[F(c_k+i\cdot)](u)\biggr)
\end{align}
in $L^2$ as $y \to \infty$. Furthermore, we find that
\begin{align*}
\MoveEqLeft \int_\mathbb{R} \biggr\vert \int_{c_1}^{c_2} e^{-(x+ iy)u}F(x+ iy)dx \biggr\vert^2 dy \\
&\leq e^{-2(c_1u \wedge c_2u)} (c_2-c_1)^2 \sup_{a<x<b}\int_\mathbb{R} \vert F(x+iy)\vert^2 dy < \infty.
\end{align*}
From these observations we deduce the existence of a sequence $(y_n)_{n\geq 1}\subseteq (0,\infty)$ with $y_n \to \infty$ such that
\begin{align*}
\int_{c_1}^{c_2} e^{-(x\pm iy_n)u}F(x\pm iy_n)dx \to 0
\end{align*}
and (\ref{L2conv}) holds along $(y_n)_{n\geq 1}$ for almost all $u\in \mathbb{R}$ as $n \to \infty$. Combining this with (\ref{contourInteg}) yields $e^{-c_1u}\mathcal{F}^{-1}[F(c_1+i\cdot)](u) = e^{-c_2u}\mathcal{F}^{-1}[F(c_2+i\cdot)](u)$ for almost all $u \in \mathbb{R}$. Consequently, there exists a function $f:\mathbb{R}\to \mathbb{C}$ with the property that $f(u) = e^{-cu}\mathcal{F}^{-1}[F(c+i \cdot)](u)$ for almost all $u \in \mathbb{R}$ for any given $c \in (a,b)$. For such $c$ we compute
\begin{align*}
\int_\mathbb{R} \vert e^{cu}f(u) \vert^2 du = \int_\mathbb{R} \vert \mathcal{F}^{-1}[F(c+i\cdot)](u) \vert^2du \leq \sup_{x \in (a,b)} \int_\mathbb{R} \vert F(x+iy)\vert^2 dy <\infty. 
\end{align*}
Consequently, $u \mapsto e^{cu}f(u) \in L^2$ for any $c \in (a,b)$ and by the monotone convergence theorem, this holds as well for $c=a$ and $c=b$. Furthermore, if $c \in (a,b)$, we can choose $\varepsilon >0$ such that $c\pm \varepsilon \in (a,b)$ as well, from which we get that
\begin{align*}
\MoveEqLeft\biggr(\int_\mathbb{R} \vert f(u) \vert e^{cu} du\biggr)^2
 \\
 &\leq \biggr(\int_0^\infty \vert f(u) e^{(c+\varepsilon)u}\vert^2du +\int_{-\infty}^0 \vert f(u) e^{(c-\varepsilon)u}\vert^2du\biggr)\int_0^\infty e^{-2\varepsilon u} du <\infty
\end{align*}
by H\"{o}lder's inequality. This shows that $\mathcal{S}_{a,b}\subseteq D(f)$. Finally, we find for $z=x+iy \in \mathcal{S}_{a,b}$ (by definition of $f$) that
\begin{align*}
\mathcal{L}[f](z) = \int_\mathbb{R} e^{iyu} e^{xu} f(u) du = \mathcal{F}[\mathcal{F}^{-1}[F(x+i\cdot)]](y) = F(z),
\end{align*}
and this completes the proof.
\end{proof}

\begin{proof}[Proof of Lemma~\ref{autoregKernel}]
Observe that, generally, $h(z)\neq 0$ if $\vert z \vert > \vert \eta \vert ([0,\infty))$, and thus, under the assumption that $h(iy) \neq 0$ and by continuity of $h$ there must be an $a<0$ such that $h(z) \neq 0$ for all $z \in \mathcal{S}_{a,0}$. The fact that $\vert h(z)\vert \sim \vert z \vert$ as $\vert z \vert \to \infty$ and, once again, the continuity of $h$ imply that \eqref{hardyCond} is satisfied for $1/h$, and thus get the existence of a function $x_0:\mathbb{R}\to \mathbb{R}$ such that $\LL [x_0]= 1/h$ on $\mathcal{S}_{a,0}$ and $t \mapsto e^{ct}x_0(t)\in L^2$ for all $c \in [a,0]$. Observe that this gives in particular that $x_0 \mathds{1}_{(-\infty,0]}\in L^1$ and thus, since $x_0\in L^2$, we also get that $x_0 \mathds{1}_{(-\infty,t]} \in L^1$ for all $t \in \mathbb{R}$.  By comparing Laplace transforms it follows that
\begin{align*}
x_0(t-r) - x_0(s-r) = \int_s^t \int_{[0,\infty)} x_0(u-v-r) \, \eta (dv) \, du + \mathds{1}_{(s,t]}(r)
\end{align*}
for almost all $r\in \mathbb{R}$ for each $s<t$. We may now let $s \to -\infty$ and use Lebesgue's theorem on dominated convergence to obtain \eqref{x0rel}.

Suppose $\eta$ has $n$-th moment for some $n\in \NN$. Then
\begin{align*}
\LL[u \mapsto u^n x_0(u)](iy) = i^{-n} D^n \frac{1}{h(iy)},
\end{align*}
where $D$ denotes the derivative, and it follows that $u \mapsto u^n x_0(u) \in L^2$, since 
\begin{align*}
\vert D^k h( iy) \vert \leq 1 + \int_{[0,\infty)} v^k \vert \eta \vert (dv)< \infty,
\end{align*}
for for $k \in \{1,\dots, n\}$. Fix $q \in [1/n,2)$. Then it follows by Hölder's inequality that
\begin{align*}
\MoveEqLeft \int_\mathbb{R} \vert x_0 (u)\vert^q\, du  \\
&\leq \biggr(\int_\mathbb{R} \big(x_0(u) (1+\vert u \vert^n) \big)^2\, du \biggr)^{q/2}\biggr(\int_\mathbb{R} (1+\vert u \vert^n)^{-2q/(2-q)}\, du \biggr)^{1-q/2} < \infty,
\end{align*}
which shows $x_0 \in L^q$ for $q \in [1/n,2)$. We have already argued that \eqref{x0rel} holds, which in particular shows that $x_0 \in L^\infty$, and we therefore get $x_0 \in L^q$ for $q \in [1/n,\infty]$.

If $\eta$ has exponential moment of order $\delta$ then we can find $a<0<b\leq \delta$ such that $1/h$ satisfies \eqref{hardyCond} and therefore, we have that $u \mapsto x_0(u)e^{cu} \in L^2$ for $c \in [a,b]$. If $h(z) \neq 0$ for all $z \in \CC$ with $\Real(z) \leq 0$ we can argue that (\ref{hardyCond}) holds with $a = -\infty$ (and $b=0$) in the same way as above and, thus, Lemma~\ref{helpLemma} implies $x_0(u) =0$ for $u<0$. 
\end{proof}
The following lemma is used to ensure uniqueness of solutions to \eqref{SDDEmodel}:
\begin{lemma}\label{forfixs}
Fix $s \in \mathbb{R}$. Suppose that $h(iy) \neq 0$ for all $y \in \RR$ and that, given $(Y_t)_{t\leq s}$, a process $(X_t)_{t \in \RR}$ satisfies 
\begin{align}\label{SDDEw/oNoise}
\begin{aligned}
X_t &= X_s + \int_s^t \int_{[0,\infty)} X_{u-v} \eta(dv) du, \quad t \geq s \\
 X_t &= Y_t, \quad t <s
 \end{aligned}
\end{align}
for each $t\in \mathbb{R}$ almost surely and has bounded first moments (that is, $\EE [\vert X_t \vert ]  \leq c$ for some $c>0$ and all $t \in \RR$). Then
\begin{align}\label{solution}
X_t = X_s x_0(t-s) + \int_s^\infty \int_{(u-s,\infty)} Y_{u-v}\,\eta (dv) \, x_0 (t-u)\,du
\end{align}
for Lebesgue almost all $t \geq s$ outside a $\mathbb{P}$-null set. 
\end{lemma}
 
 \begin{proof}
%Since $X$ has bounded first moment, we can choose a set of $\PP$ measure zero outside of which 
%\begin{align*}
%\int_s^\infty e^{-bt} \vert X_t \vert \, dt < \infty \quad \text{and} \quad  \int_s^\infty  e^{-bt} \int_s^t \int_{-\infty}^0 \vert X_{u+v} \vert \, \vert \eta \vert (dv) \,du \, dt < \infty
%\end{align*} 
%for any $b>0$. 

Observe, by Fubini's theorem, we can remove a $\mathbb{P}$-null set and have that (\ref{SDDEw/oNoise}) is satisfied for Lebesgue almost all $t\in \mathbb{R}$. Let $a<0$ be such that $h(z) \neq 0$ for all $z \in \mathcal{S}_{a,0}$ (this is possible due to the assumption $h(iy) \neq 0$ for all $y \in \RR$). For such $z$, 
 \begin{align*}
\LL[ X \mathds{1}_{[s,\infty)}](z)  = & \LL \left[\mathds{1}_{[s,\infty)}\left\{ X_s + \int_s^\cdot \int_{[0,\infty)} X_{u-v} \,\eta(dv)\, du\right\}\right](z) \\
=& -\frac{X_s  e^{zs} }{z} +\int_s^\infty e^{zt} \int_s^t \int_{[0,\infty)} X_{u-v}\, \eta(dv)\, du\, dt \\
%=&-\frac{X_s  e^{zs} }{z}+\int_{[0,\infty)} \int_s^\infty e^{zt} \int_{s-v}^{t-v}  X_u \, du\, dt\, \eta(dv) \\
= & -\frac{1}{z} \left(X_s  e^{zs}  +\int_{[0,\infty)} \int_{s-v}^\infty X_u  e^{z(u+v)}  \, du\,\eta(dv)\right) \\
%=&  -\frac{1}{z} \left(X_s  e^{zs}  +\LL[\eta](z)\LL [X\mathds{1}_{[s,\infty)}](z)\right. \\
%& \left. + \int_{[0,\infty)} \int_{s-v}^s Y_u  e^{z(u+v)}  \, du\,\eta(dv) \right)  \\
 = & -\frac{1}{z} \Biggr( X_s  e^{zs}+\LL[\eta](z)\LL [X\mathds{1}_{[s,\infty)} ](z) \\
 & \left.+ \LL \left[ \mathds{1}_{[s,\infty)}  \int_{( \cdot-s,\infty)}Y_{\cdot-v}\,\eta(dv)\right](z) \right).
 \end{align*}
 This gives that
 \begin{align}\label{lapofX}
 \LL[ X \mathds{1}_{[s,\infty)}](z)  &= \frac{ X_s e^{zs} }{h(z) } + \frac{  \LL \left[ \mathds{1}_{[s,\infty)}  \int_{( \cdot-s,\infty)}Y_{\cdot-v}\,\eta(dv)\right](z)}{h(z)}.
 \end{align}
Since $Y$ has bounded first moments, $\int_s^\infty \int_{(u-s,\infty)} \vert Y_{u-v} \vert  \, \vert \eta \vert  (dv) \, \vert  x_0 (t-u)\vert  \,du$ is almost surely finite by Lemma~\ref{autoregKernel}, and Fubini's theorem gives that it is finite for Lebesgue almost all $t\in \mathbb{R}$ outside a $\PP$-null set. Furthermore, again by Lemma \ref{autoregKernel}, there exists $\varepsilon>0$ such that
\begin{align*}
\int_\RR e^{-\varepsilon t} \int_s^\infty \vert x_0  (t-u) \vert\, du\, dt = \int_\RR e^{-\varepsilon t}  \vert x_0  (t) \vert \,dt \int_s^\infty e^{-\varepsilon u}\,du < \infty. 
\end{align*}
From this it follows that almost surely $ \int_s^\infty \int_{(u-s,\infty)} Y_{u-v}\,\eta (dv) \, x_0 (t-u)\,du$ is well-defined and that its Laplace transform exists on $\mathcal{S}_{-\varepsilon,0}$. We conclude that 
\begin{align*}
\LL \left[ \int_s^\infty \int_{(u-s,\infty)} Y_{u-v}\,\eta (dv) \, x_0 (\cdot-u)\,du \right] (z) =  \frac{  \LL \left[ \mathds{1}_{[s,\infty)}  \int_{( \cdot-s,\infty)}Y_{\cdot-v}\,\eta(dv)\right](z)}{h(z)}, 
\end{align*}
for $z \in \mathcal{S}_{-\varepsilon,0}$, and the result follows since we also have $\LL [x_0 (\cdot -s )] (z) = e^{zs} /h(z)$ for $z \in \mathcal{S}_{-\varepsilon,0}$.
\end{proof}

\begin{proof}[Proof of Theorem~\ref{psiCharacterization}]
We start by noting that if $(X_t)_{t\in \mathbb{R}}$ and $(Y_t)_{t\in \mathbb{R}}$ are two integrable measurable solutions to (\ref{SDDEmodel}) then, for fixed $s \in \mathbb{R}$,
\begin{align}\label{detDE}
\MoveEqLeft U_t = U_s + \int_s^t \int_{[0,\infty)} U_{u-v}\, \eta(dv) \,du
\end{align}
almost surely for each $t \in \mathbb{R}$, when we set $U_t := X_t - Y_t$. Particularly, for a given $t\in \mathbb{R}$, we get by Lemma~\ref{forfixs},
\begin{align}\label{uniqueRelation}
U_r = U_s x_0(r-s) + \int_s^\infty \int_{(u-s,\infty)} U_{u-v}\,\eta (dv) \, x_0 (r-u)\,du
\end{align}
for Lebesgue almost all $r>t-1$ and all $s \in \mathbb{Q}$ with $s\leq t-1$. For any such $r$ we observe that the right-hand side of (\ref{uniqueRelation}) tends to zero in $L^1(\mathbb{P})$ as $\mathbb{Q}\ni s \to -\infty$, from which we deduce $U_r = 0$ or, equivalently, $X_r = Y_r$ almost surely. By continuity of $(U_r)_{r\in \mathbb{R}}$ in $L^1(\mathbb{P})$ (see e.g., \cite[Corollary~A.3]{QOU}), we get that $X_t = Y_t$ almost surely as well. This shows that a solution to (\ref{SDDEmodel}) is unique up to modification.

We have $\mathbb{E}[\vert Z_u \vert ]\leq a + b \vert u  \vert$ for any $u,v \in \mathbb{R}$ with suitably chosen $a, b >0$ (see \cite[Corollary~A.3]{QOU}), and this implies that 
\begin{align*}
\MoveEqLeft \EE \biggr[\int_\RR \vert Z_u \vert \int_{[0,\infty)} \vert x_0 (t-u-v) \vert \, \vert \eta \vert (dv) du\biggr] \\
 \leq & a \vert \eta \vert ([0,\infty)) \int_\RR \vert x_0 (u) \vert \,  du +  b \int_\RR \vert u \vert \int_{[0,\infty)} \vert x_0 (t-u-v) \vert \, \vert \eta \vert (dv)\, du \\
\leq & \left( a \vert \eta \vert ([0,\infty)) + b \int_{[0,\infty)} v \vert \eta \vert (dv) \right) \int_\RR \vert x_0 (u) \vert \,  du  \\
& +  b \vert \eta \vert ([0,\infty)) \int_\RR( \vert t \vert + \vert u \vert ) \vert x_0 (u) \vert \, du.
\end{align*}
This is finite by Lemma~\ref{autoregKernel} and Corollary~\ref{cadlagBounded}, and $\int_\RR Z_u \int_{[0,\infty)}  x_0 (t-u-v) \,  \eta (dv) \,du$ is therefore almost surely well-defined.

%\begin{align}\label{solutionGoal}
%\int_s^t \int_{[0,\infty)} X_{u-v}\eta(dv)du = \int_\mathbb{R}Z_u \int_{[0,\infty)} [x_0(t-u-v)-x_0(s-u-v)]\, \eta(dv) \,du.
%\end{align}

To argue that $X_t = Z_t + \int_\RR Z_u \int_{[0,\infty)}  x_0 (t-u-v) \,  \eta (dv) \,du$, $t \in \RR$, satisfies (\ref{SDDEmodel}), let $s<t$ and note that by Lemma \ref{autoregKernel} we have
\begin{align*}
\MoveEqLeft\int_s^t \int_{[0,\infty)} X_{u-v}\,\eta(dv)\,du-\int_s^t\int_{[0,\infty)} Z_{u-v}\,\eta(dv)\, du \\
=& \int_s^t \int_{[0,\infty)} \int_{\mathbb{R}} Z_r \int_{[0,\infty)} x_0 (u-v-r-w)\,\eta(dw)\,dr\,\eta(dv)\,du \\
=& \int_\mathbb{R} Z_r \int_{[0,\infty)} \int_{s-r-w}^{t-r-w} \int_{[0,\infty)} x_0 (u-v)\,\eta(dv)\,du\,\eta(dw)\,dr \\
=&\int_\mathbb{R} Z_r \int_{[0,\infty)} [x_0(t-r-w)-x_0(s-r-w)]\,\eta(dw)\,dr \\
-& \int_\mathbb{R}\int_{[0,\infty)} Z_r [\mathds{1}_{[0,\infty)}(t-r-w)-\mathds{1}_{[0,\infty)}(s-r-w)]\,\eta(dw)\,dr \\
=& \int_\mathbb{R} Z_r \int_{[0,\infty)} [x_0(t-r-w)-x_0(s-r-w)]\,\eta(dw)\, dr \\
-& \int_s^t \int_{[0,\infty)} Z_{r-w}\,\eta(dw)\,dr.
\end{align*}
Next, we write
\begin{align}\label{incRep}
X_t = \int_\mathbb{R} (Z_t-Z_{t-u})\int_{[0,\infty)} x_0 (u-v)\,\eta(dv)\,du, \quad t \in \mathbb{R},
\end{align}
using Lemma \ref{autoregKernel}. Since $(Z_t)_{t\in \mathbb{R}}$ is continuous in $L^1(\mathbb{P})$, one shows that the process
\begin{align*}
X^n_t := \int_{-n}^n (Z_t-Z_{t-u}) \int_{[0,\infty)} x_0 (u-v)\,\eta(dv)\,du, \quad t \in \mathbb{R},
\end{align*}
is stationary by approximating it by Riemann sums in $L^1(\mathbb{P})$. Subsequently, due to the fact that $X^n_t \to X_t$ almost surely as $n \to \infty$ for any $t \in \mathbb{R}$, we conclude that $(X_t)_{t \in \mathbb{R}}$ is stationary. This completes the proof.
%Finally, it follows from Lemma~\ref{autoregKernel} that if $h(z) \neq 0$ for all $z \in\CC$ with $\Real(z) \leq 0$, then $x_0(t)=0$ for $t<0$. As a consequence,
%\begin{align*}
%X_t = Z_t + \int_\mathbb{-\infty}^t Z_u\int_{[0,\infty)} x_0(t-u-v)\eta (dv)du,
%\end{align*}
%from which the causality property is immediate.
\end{proof}

\begin{proof}[Proof of Corollary \ref{simmaNoise}]

It follows from \eqref{minusOne} and Corollary~\ref{cadlagBounded} that
\begin{align*}
 \MoveEqLeft Z_t + \int_\mathbb{R}Z_{t-u} \int_{[0,\infty)} x_0 (u-v) \, \eta (dv) \, du \\
 &=  \int_\mathbb{R} [Z_{t-u} - Z_t] \int_{[0,\infty)} x_0 (u-v) \, \eta (dv) \, du \\
&=   \int_\RR \int_\mathbb{R}  [\theta(t-u-r) - \theta(t-r)] [ x_0 (du) - \delta_0(du) ] dL_r \\
&=   \int_\RR \int_\mathbb{R}  \theta(t-u-r) \, x_0 (du) \,  dL_r \\
&=  \int_\RR\theta \ast x_0(t-r) \,  dL_r 
\end{align*}
where we have used that $\int_\RR x_0 (du) =0$ since $x_0(t) \to 0$ for $t \to \pm \infty$ by \eqref{x0rel}.
\end{proof}

\begin{proof}[Proof of Theorem~\ref{levelResult}]
First, observe that there exists $\psi: \mathbb{R}\to \mathbb{R}$ such that $\sum_{k=0}^n \theta \ast\phi^{\ast k } \to \psi$ in $L^2$ as $n \to \infty$. To see this, one can use that $L^2$ is complete and that the Fourier transform is an $L^2$ isometry. Moreover, since $\vert \phi\vert ((-\infty,0])=0$ and $\theta(t)=0$ for $t<0$, $\psi (t) = 0$ for Lebesgue almost all $t<0$

Suppose now that we have a square integrable stationary solution $(X_t)_{t\in \mathbb{R}}$. Then, using a stochastic Fubini (see \cite[Theorem~3.1]{QOU}), it follows that for each $t\in \mathbb{R}$ almost surely,
\begin{align}
\begin{aligned}
X_t &=X \ast \phi^{\ast n} (t) + \sum_{k=0}^{n-1} (\theta \ast L)\ast \phi^{\ast k} (t) \\
&= X \ast \phi^{\ast n} (t) + \biggr(\sum_{k=0}^{n-1} \theta \ast \phi^{\ast k}\biggr)\ast L (t)
\end{aligned} \label{xRepBackward}
\end{align}
for an arbitrary $n \in \mathbb{N}$. (For convenience, we use the notation $f\ast L (t) = \int_\mathbb{R}f(t-u)\, dL_u$ for $f\in L^2$.) By Jensen's inequality,
\begin{align*}
\mathbb{E}[X \ast \phi^{\ast n}(t)^2] \leq \vert \phi \vert (\mathbb{R})^{2n} \mathbb{E}[X_0^2] \to 0
\end{align*}
as $n \to \infty$, and it therefore follows from (\ref{xRepBackward}) that
\begin{align*}
\int_\mathbb{R} \sum_{k=0}^{n-1} \theta \ast \phi^{\ast k}(t-u)dL_u = \biggr( \sum_{k=0}^{n-1} \theta \ast \phi^{\ast k}\biggr) \ast L (t) \to X_t
\end{align*}
in $L^2(\mathbb{P})$ as $n \to \infty$. Since $\sum_{k=0}^{n-1} \theta \ast \phi^{\ast k}(t-\cdot) \to\psi (t-\cdot)$ in $L^2$, it follows by isometry that $X_t = \int_\mathbb{R} \psi (t-u)dL_u$ almost surely. 

Conversely, if one defines a square integrable stationary process $(X_t)_{t\in \mathbb{R}}$ by $X_t = \psi\ast L (t)$, $t \in \mathbb{R}$, we get that
\begin{align*}
\MoveEqLeft X_t - \theta\ast L (t) \\
&= \lim_{n \to \infty} \biggr(\sum_{k=1}^n \theta \ast \phi^{\ast k} \biggr)\ast L (t)
= \lim_{n\to \infty} \biggr(\sum_{k=0}^{n-1}\theta \ast \phi^{\ast k} \ast L \biggr)\ast \phi (t) = X\ast \phi (t)
\end{align*}
almost surely, where the limits are in $L^2(\mathbb{P})$. Thus, $(X_t)_{t\in \mathbb{R}}$ satisfies (\ref{levelModel}).
\end{proof}
\subsection*{Acknowledgments}
The research was supported by the Danish Council for Independent Research (Grant DFF - 4002 - 00003).

\bibliographystyle{chicago}

\end{document}